\documentclass[12pt,reqno]{amsart}

\usepackage{amssymb}

\newtheorem{lemma}{Lemma}
\newtheorem{theorem}{Theorem}
\newtheorem{corollary}{Corollary}

\theoremstyle{remark}
\newtheorem*{remark}{Remark}

\title[An improved upper bound for the size of the multiplicative 3-Sidon sets]%
  {An improved upper bound for the size of the multiplicative 3-Sidon sets}

\author[P\'eter P\'al Pach]{P\'eter P\'al Pach$^\dag$}
\thanks{${}^\dag$ Partially supported by the National Research, Development and Innovation Office NKFIH
  (Grant Nr.~PD115978) and the J\'anos Bolyai Research
  Scholarship of the Hungarian Academy of Sciences. The author has also received funding from the European Research Council (ERC) under the European Union’s Horizon 2020 research and innovation programme (grant agreement No 648509). This publication reflects only its author's view; the European Research Council Executive Agency is not responsible for any use that may be made of the information it contains.}
\email{ppp@cs.bme.hu}
\address{Department of Computer Science and Information Theory, Budapest
  University of Technology and Economics, 1117 Budapest, Magyar tud\'osok
  k\"or\'utja 2., Hungary,  \newline 
  Department of Computer Science and DIMAP, University of Warwick, Coventry CV4 7AL, UK.}

\begin{document}
\baselineskip=16pt

\begin{abstract}
We say that a set is a multiplicative 3-Sidon set if the equation $s_1s_2s_3=t_1t_2t_3$ does not have a solution consisting of distinct elements taken from this set. In this paper we show that the size of a multiplicative 3-Sidon subset of $\{1,2,\dots,n\}$ is at most $\pi(n)+\pi(n/2)+n^{2/3}(\log n )^{2^{1/3}-1/3+o(1)}$, which improves the previously known best bound $\pi(n)+\pi(n/2)+cn^{2/3}\log n/\log\log n$.
\end{abstract}

\maketitle

\section{Introduction}

A set  $A\subseteq\mathbb{N}$ is called a Sidon set, if for every $l$ the equation $x+y=l$ has at most one solution with $x,y \in A$. A multiplicative Sidon set is analogously defined by requiring that the equation $xy=l$ has at most one solution in $A$. To emphasize the difference, throughout the paper the first one will be called an additive Sidon set. There are many results on the maximal size of an additive Sidon set in $\{1,2,\dots,n\}$ and on the infinite case, as well.  Moreover, a natural generalization of additive Sidon sets is also studied,  they are called  $B_h[g]$ sequences: A sequence $A$ of positive integers is called a $B_h[g]$ sequence, if
every integer $n$ has at most $g$ representations $n = a_1 + a_2 + \dots + a_h$ with all
$a_i$ in $A$ and $a_1 \leq a_2 \leq\dots\leq a_h$. Note that an additive Sidon sequence is a $B_2[1]$ sequence. 

In this paper a set $A\subseteq\mathbb{N}$ is going to be called a multiplicative $k$-Sidon set, if the equation $s_1s_2\dots s_k=t_1t_2\dots t_k$ does not have a solution in $A$ consisting of distinct elements. 

In \cite{P_Acta}  the equation $s_1s_2\dots s_k=t_1t_2\dots t_l$ was investigated and it was proved that for $k\ne l$ there is no density-type theorem, which means that a subset of $\{1,2,\dots, n\}$ not containing a ``nontrivial solution'', that is, a solution consisting of distinct elements, can have size $c\cdot n$. However, a Ramsey-type theorem can be proved: If we colour the integers by $r$ colours, then the equation $a_1a_2\dots a_k=b_1b_2\dots b_l$ has a nontrivial monochromatic solution. The case when $k=l$ is even more interesting, in this paper this is going to be investigated.

Let $G_k(n)$ denote the maximal size of a multiplicative $k$-Sidon set in $\{1,2,\dots,n\}$. Erd\H{o}s studied the case $k=2$.  In \cite{Erdos38} he gave a construction with  $\pi(n)+c_1n^{3/4}/(\log{n})^{3/2}$ elements, and proved that the maximal size of such a set is at most $\pi(n)+c_2n^{3/4}$.  31 years later  Erd\H{o}s \cite{Erdos69} himself improved this upper bound to $\pi(n)+c_2n^{3/4}/(\log n)^{3/2}$. Hence, in the lower- and upper bounds for $G_2(n)$ not only the main terms are the same, but the error terms only differ in a constant factor. In \cite{pach} it was shown that 
\begin{equation}\label{regi3ra}
\pi(n)+\pi(n/2)+c_1n^{2/3}/(\log n)^{4/3} \leq G_3(n)\leq \pi(n)+\pi(n/2)+c_2n^{2/3}\frac{\log n}{\log\log n}.
\end{equation} 
In this paper our aim is to improve the upper bound by showing that 
\begin{theorem}\label{3raujtetel} 
\begin{equation}\label{ujG3}
G_3(n)\leq \pi(n)+\pi(n/2)+n^{2/3}(\log n)^{2^{1/3}-1/3+o(1)}.
\end{equation}
\end{theorem}
\noindent
Note that $2^{1/3}-1/3\approx 0.9266$.
 
Our question about the solvability of $a_1a_2\dots a_k=b_1b_2\dots b_k$ is not only a natural generalization of the multiplicative Sidon sequences, but it is also strongly connected to the following problem from combinatorial number theory: Erd\H{o}s, S\'ark\"ozy and T. S\'os \cite{ErdSarSos} examined how many elements of the set
$\{1,2,\dots,n\}$ can be chosen in such a way that none of the $2k$-element products from this set is a perfect square. The maximal size of such a subset is denoted by $F_{2k}(n)$. Note that the functions $F$ and  $G$ satisfy the inequality $F_{2k}(n)\leq G_k(n)$ for every $n$ and $k$ because if the equation $a_1\dots a_k=b_1\dots b_k$ has a solution of distinct elements, then the product of these $2k$ numbers is a perfect square. Erd\H{o}s, S\'ark\"ozy and T. S\'os proved the following estimates for $k=3$: 
\begin{equation*}
\pi(n)+\pi\left({n}/{2}\right)+c_1\frac{n^{2/3}}{(\log n)^{4/3}} \leq F_6(n)
\leq\pi(n)+\pi\left( {n}/{2}\right)+c_2n^{7/9}\log n.
\end{equation*}
Besides, they noted that by improving their graph theoretic lemma used in the proof the upper bound $\pi(n)+\pi(n/2)+cn^{2/3}\log n$ could be obtained, so the lower and upper bounds would only differ in a log-power factor in the error term. Later Gy\H{o}ri \cite{gyori} improved this graph theoretic lemma and gained the desired bound.

In \cite{pach} the error term of the upper bound for $F_6(n)$ was improved by a $(\log\log n)$-factor as a consequence of \eqref{regi3ra}. Now, in this paper this error term is going to be further improved by a $(\log n)$-power factor, namely, \eqref{ujG3} implies that:
\begin{corollary}
$$F_6(n)\leq \pi(n)+\pi(n/2)+n^{2/3}(\log n)^{2^{1/3}-1/3+o(1)}.$$
\end{corollary}


\section{Preliminary lemmata}

Throughout the paper the maximal number of edges of a graph not containing a cycle of length $k$ is conventionally denoted by $ex(n,C_k)$, and let us use the notation $ex(u,v,C_{2k})$ for the maximal number of edges of a $C_{2k}$-free bipartite graph, where the number of vertices of the two classes are $u$ and $v$. (Note that every graph appearing in this paper is simple.)

Throughout the paper the number of prime factors of $n$ with multiplicity is going to be denoted by $\Omega(n)$.

\begin{lemma}\label{ujuvlemma}
Let $n$ be positive integer. Every $m\leq n$ can be written in the form
$$m=uv\ (u,v\in\mathbb{N}),$$
where one of the following conditions hold:
\begin{itemize}
\item[(i)] $n^{2/3}<u$ is a prime;
\item[(ii)] $u,v\leq n^{2/3}$ such that $2\Omega(u)-2\leq \Omega(v)$.  
\end{itemize}

\end{lemma}

\begin{proof}
Let $m=p_1p_2\dots p_r$, where $p_1\geq p_2\geq\dots\geq p_r$ are primes. 
If $p_1>n^{2/3}$, then $u=p_1$ and $v=m/u$ is an appropriate choice. From now on, let us assume that $p_1\leq n^{2/3}$.
Let $i$ be the smallest index such that $u=p_1p_2\dots p_i\geq m^{1/3}$. It is clear that $v:=m/u\leq m^{2/3}\leq n^{2/3}$, we show that $u\leq n^{2/3}$ also holds. Otherwise, $p_1p_2\dots p_{i-1}<m^{1/3}$ and $p_1p_2\dots p_i>m^{2/3}$ together imply that $p_i>m^{1/3}$, hence $i\leq 2$. If $i=1$, then $u=p_1\leq n^{2/3}$. Finally, if $i=2$, then $p_1<m^{1/3}$ yields $u=p_1p_2<m^{2/3}\leq n^{2/3}$.

As $p_1p_2\dots p_{i-1}<m^{1/3}$, we have $3(i-1)< r$, therefore, $\Omega(v)\geq 2\Omega(u)-2$.

\end{proof}

\begin{lemma}\label{exC6}
Let $n\in\mathbb{N}$. Then
$$ex(n, C_6) <  n^\frac{4}{3},$$
if $n$  is large enough.
\end{lemma}
\begin{proof}
According to the second statement of Theorem 1.1 in \cite{furedi} the stronger inequality $ex(n, C_6) < 0.6272 n^\frac{4}{3}$ also holds. 
\end{proof}

\begin{lemma}\label{exC6uv}
Let $u, v \in\mathbb{N}$. Then
$$ex(u, v, C_6) \leq 2^{1/3}(uv)^{2/3}+16(u+v).$$
\end{lemma}
\begin{proof}
This is Theorem 1.2 in \cite{furedi}. 
\end{proof}

\begin{lemma}\label{gy1} Let $u,v\in\mathbb{N}$ satisfying $v\leq u$. Then
$$ex(u,v,C_6) < 2u + v^{2}/2.$$
\end{lemma}
\begin{proof}
This is Theorem 1. in \cite{gyori}. 
\end{proof}



\begin{lemma}\label{Omegabecs}
Let us denote by $N_i(x)$ the number of positive integers
$n \leq x$ satisfying $\Omega(n) \leq i$ and by $M_i(x)$ the number of positive integers
$n \leq x$ satisfying $\Omega(n) \geq i$.

For every $\delta > 0$ there exists some
constant $C = C(\delta)$ such that for $1 \leq i \leq (1 - \delta) \log \log x$ we have
$N_i(x) < C(\delta) \cdot \frac{x}{\log x}\cdot \frac{(\log\log x)^{i-1}}{(i-1)!}$
and for $(1 + \delta) \log \log x \leq i \leq (2 - \delta) \log \log x$ we have 
$M_i(x) < C(\delta) \cdot \frac{x}{\log x}\cdot \frac{(\log\log x)^{i-1}}{(i-1)!}$
%
\end{lemma}

\begin{proof}
The first statement is Lemma 2.8. in \cite{pach}, the second statement is a direct consequence of Corollary 1. in \cite{erdsar}.
\end{proof}
\begin{remark}
Let $\alpha:=\frac{i-1}{\log\log x}$, then we have 
$$\frac{(\log\log x)^{i-1}}{(i-1)!}=\frac{(\log\log x)^{\alpha \log\log x}}{(\alpha \log\log x)!}\leq \left( \frac{e}{\alpha}\right)^{\alpha\log\log x}=(\log x)^{\alpha-\log\alpha}.$$
Note that when we apply Lemma~\ref{Omegabecs}. it is going to use that $C(\delta) \cdot \frac{x}{\log x}\cdot \frac{(\log\log x)^{i-1}}{(i-1)!}\leq C(\delta)\cdot x(\log x)^{\alpha-\alpha\log\alpha-1}$.
\end{remark}

\section{Proof of Theorem~\ref{3raujtetel}}


Let us assume that for $A\subseteq \{1,2,\dots,n\}$ the equation 
\begin{equation}\label{abcdef2}
s_1s_2s_3=t_1t_2t_3\  (s_1,s_2,s_3,t_1,t_2,t_3\in A)
\end{equation}
has no solution consisting of distinct elements.

Let 
$A=\{a_1,\dots,a_l\},\text{ where }1\leq a_1<a_2<\dots<a_l\leq n.$ Applying Lemma 
  \ref{ujuvlemma}.  we obtain that the elements of the set $A$ can be written in the form
$a_i=u_iv_i,$
where $u_i$ and $v_i$  are positive integers and one of the following conditions holds:
\begin{itemize}
\item[(i)] $n^{2/3}<u_i$ is a prime,
\item[(ii)] $u_i,v_i\leq n^{2/3}$ and $\Omega(v_i)\geq 2\Omega(u_i)-2$.
\end{itemize}
If an element $a_i$ can be written as $u_iv_i$ in more than one appropriate way, then we choose such a representation $a_i=u_iv_i$  where 
$v_i$ is minimal. The number of those elements of  $A$ for which  $u_i=v_i$   can be estimated 
from above by the number of square numbers in  $\{1,2,\dots,n\}$, hence
\begin{equation}
|\{i\ |\ 1\leq i\leq l, u_i=v_i\}|\leq \sqrt{n}.
\end{equation}
As $\sqrt{n}$ is negligible compared to the error term $n^{\frac{2}{3}}(\log n)^{2^{1/3}-1/3+o(1)}$, it suffices to prove the theorem for sets which does not contain squares. From now on, let us assume that $u_i \ne v_i$  for every
$a_i\in A$. 

Assume  that  \eqref{abcdef2} has no such solution
where
$s_1,s_2,s_3,t_1,t_2,t_3$ are  distinct. Let $G=(V,E)$ be a
graph where the vertices are the integers not greater than $n^{2/3}$
and the primes from the interval
$(n^{2/3},n]$:
$$V(G)=\{a\in\mathbb{N}\ |\ a\leq n^{2/3}\}\cup\{p\ |\ n^{2/3}<p\leq n, p\text { is a prime}\}.$$
Then the number of the vertices of $G$ is
$|V(G)|=\pi(n)+[n^{2/3}]-\pi(n^{2/3}).$
The edges of $G$ are defined in such a way that they correspond to the elements of $A$:
For each $1\leq i\leq l$ let $u_iv_i$  be an edge, and denote it by  $a_i$:
$E(G)=\{u_iv_i\ |\ 1\leq i\leq l\}.$
In this way distinct edges are assigned to distinct elements of $A$. In the graph
$G$ there are  no loops because 
we have omitted the elements where $u_i=v_i$, moreover 
$|E(G)|=|A|=l.$
Furthermore, $G$ contains no hexagons. Indeed, if
$x_1x_2x_3x_4x_5x_6x_1$  is a hexagon in $G$, then 
$$s_1=x_1x_2,t_1=x_2x_3,s_2=x_3x_4,t_2=x_4x_5,s_3=x_5x_6,t_3=x_6x_1$$
would be a solution of  \eqref{abcdef2}, contradicting our assumption.

Now our aim is to estimate from above the number of edges of $G$. At first let us partition the edges of  $G$ into some parts.
Let $G_0$ be the subgraph that contains such edges $u_iv_i$ of $G$ for which $\max(u_i,v_i)\leq \sqrt{n}$:
$$E(G_0)=\{u_iv_i\ |\ u_i,v_i\leq \sqrt{n}\}.$$

Those remaining edges that satisfy (ii) are divided into $K=\left\lfloor\frac{\log n}{6}\right\rfloor$ parts. For these 
edges $u_iv_i$ both $\sqrt{n}< \max(u_i,v_i)\leq n^{2/3}$ and $2\Omega(u_i)-2\leq \Omega(v_i)$ hold. For
$1\leq h\leq K$ let $G_{h}$  be the subgraph which contains such edges $u_iv_i$ of  the graph 
 $G\setminus G_0$ which satisfy the inequality 
$n^{\frac{1}{2}+\frac{h-1}{6K}}< \max(u_i,v_i)\leq n^{{\frac{1}{2}}+\frac{h}{6K}}.$
The edges of the graph $G_h$ are partitioned into two classes depending on the sizes of $u_i$ and $v_i$: 
$$E(G'_{h})=\{u_iv_i\ |\ n^{\frac{1}{2}+\frac{h-1}{6K}}< u_i\leq n^{\frac{1}{2}+\frac{h}{6K}}\text { and } 2\Omega(u_i)-2\leq \Omega(v_i)\text\}\setminus E(G_0)$$
and 
$$E(G''_{h})=\{u_iv_i\ |\ n^{\frac{1}{2}+\frac{h-1}{6K}}< v_i\leq n^{\frac{1}{2}+\frac{h}{6K}}\text { and } 2\Omega(u_i)-2\leq \Omega(v_i)\text\}\setminus E(G_0).$$
 
 Finally, let $G_{K+1}$ be the graph which is obtained by deleting the edges of $G_0, G_1, \dots,G_{K}$ from $G$.
For the edges $u_iv_i$ in $G_{K+1}$  we have
 $n^{2/3}< u_i.$ That is, $u_i$ is a prime, and these edges satisfy (i):
$$E(G_{K+1})=\{u_iv_i\ |\ n^{2/3}< u_i\text{ and }u_i\text{ is  prime}\}.$$
So we divided the graph $G$ into $2K+2$ parts.

Denote by $l_h, l_h', l_h''$ the number of edges of $G_h, G_h', G_h''$, respectively ($0\leq h\leq K+1$).
In the remaining part of the proof we estimate the   number of edges $l_h$ separately, and at the end  we add up these estimates. 
 There are at most  $[n^{1/2}]$
vertices of $G_0$ that are endpoints of some edges because
 $u_iv_i\in E(G_0)$  implies $v_i<u_i\leq n^{1/2}$. Hence, by
Lemma~\ref{exC6}. for large enough $n$
\begin{equation}\label{k0}
l_0\leq (n^{1/2})^{4/3}=n^{2/3} 
\end{equation}
holds.

Now let $1\leq h\leq K$. If $a_i=u_iv_i$ is an edge of $G'_{h}$, then
$$n^{\frac{1}{2}+\frac{h-1}{6K}}< u_i\leq n^{\frac{1}{2}+\frac{h}{6K}},$$
and so
$$v_i=\frac{a_i}{u_i}\leq\frac{n}{u_i}\leq n^{\frac{1}{2}-\frac{h-1}{6K}}.  $$
For brevity, let $H=G_h', \alpha={\frac{1}{2}+\frac{h}{6K}}, \beta={\frac{1}{2}-\frac{h-1}{6K}}$. Then for every edge $uv$ in $H$ we have $u\leq n^{\alpha},v\leq n^{\beta}$ and 
\begin{equation}
2\Omega(u)-2\leq \Omega(v).
\end{equation}
Moreover, $\alpha,\beta\in [1/3,2/3]$ and $\alpha+\beta=1+\frac{1}{6K}\approx 1+\frac{1}{\log n}$.
Now we partition the edges of the bipartite graph $H$ into several subgraphs. Let $H_1$ and $H_2$ be the subgraphs containing the edges $uv$ satisfying $\Omega(u)\leq 0.55\log\log n$ and $\Omega(v)\geq 1.6\log\log n$, respectively. For the remaining edges we have $0.55\log\log n<\Omega(u)$ and $\Omega(v)\leq 1.6\log\log n$. For every 
\begin{equation}\label{klrest}
0.55\log\log n\leq k,\quad 2k-2\leq l,\quad l\leq 1.6\log\log n
\end{equation} let $H_{k,l}$ contain the edges $uv$ for which $\Omega(u)=k, \Omega(v)=l$.

Note that the graphs $H_1,H_2,H_{k,l}$ are all $C_6$-free bipartite graphs. Now, we are going to give upper bounds for the number of edges in these graphs. As a first step from all these graphs we delete the vertices with degree 0.

In $H_1$ for the two independent vertex classes we have 
$$U_1\subseteq\{u\ |\ u\leq n^{\alpha}, \Omega(u)\leq 0.55\log\log n\}\text{ and }V_1\subseteq\{v\ |\ v\leq n^{\beta}\}.$$
According to Lemma~\ref{Omegabecs}. we have $|U_1|\leq c_{1} n^{\alpha} (\log n)^{0.55-0.55\log0.55-1}<n^{\alpha}(\log n)^{-0.12}$ for some constant $c_1$ and sufficiently large $n$. Clearly, $|V_1|\leq n^\beta$. Therefore, by Lemma~\ref{exC6uv}. the number of edges of $H_1$ is at most
\begin{multline}\label{H1}
2^{1/3}(|U_1|\cdot|V_1|)^{2/3}+16(|U_1|+|V_1|)\leq 2^{1/3} n^{\frac{2}{3}(\alpha+\beta)}(\log n)^{-0.08}+16(n^\alpha+n^\beta)\leq \\
\leq c_{2} \frac{n^{2/3}}{(\log n)^{0.08}} +16(n^\alpha+n^\beta),
\end{multline}
where $c_2>2^{1/3}e^{2/3}$ is arbitrary and $n$ is large enough.

Similarly, in $H_2$ the two independent vertex classes are 
$$U_2\subseteq\{u\ |\ u\leq n^{\alpha}\},V_2\subseteq\{v\ |\ v\leq n^{\beta}\text{ and } \Omega(v)\geq 1.6\log\log n\}.$$
According to Lemma~\ref{Omegabecs}. we have $|V_2|\leq c_{3} n^{\beta} (\log n)^{1.6-1.6\log1.6-1}<(\log n)^{-0.12}$ and clearly, $|U_2|\leq n^\alpha$. Therefore, by Lemma~\ref{exC6uv}. the number of edges of $H_2$ is at most
\begin{multline}\label{H2}
2^{1/3}(|U_2|\cdot|V_2|)^{2/3}+16(|U_2|+|V_2|)\leq 2^{1/3} n^{\frac{2}{3}(\alpha+\beta)}(\log n)^{-0.08}+16(n^\alpha+n^\beta)\leq \\
c_{2} \frac{n^{2/3}}{(\log n)^{0.08}} +16(n^\alpha+n^\beta), 
\end{multline}
if $n$ is large enough.

Now, let us consider the $H_{k,l}$ graphs. Note that $k$ and $l$ satisfy \eqref{klrest} which implies that $k\leq l/2+1\leq 0.8\log\log n+1\leq 0.81\log\log n$ and $l\geq 2k-2\geq 1.1\log\log n-2\geq 1.09\log\log n$.
For the two vertex classes we have 
$$U_{k,l}\subseteq\{u\  |\ u\leq n^{\alpha}\text{ and } \Omega(u)=k\},V_{k,l}\subseteq\{v\ |\ v\leq n^{\beta}\text{ and } \Omega(v)=l\}.$$
According to Lemma~\ref{Omegabecs}. there is a $c_4>0$ not depending on $k,l,\alpha,\beta$ such that
$$|U_{k,l}|\leq c_4\cdot\frac{n^\alpha}{\log n} \cdot\frac{(\log\log n)^{k-1}}{(k-1)!}$$
and 
$$|V_{k,l}|\leq c_4\cdot \frac{n^\beta}{\log n} \cdot\frac{(\log\log n)^{l-1}}{(l-1)!}.$$
Let $d=\max\limits_{0.55\log\log n\leq k, 2k-2\leq l, l\leq 1.6\log\log n} \frac{(\log\log n)^{k-1}}{(k-1)!}\cdot \frac{(\log\log n)^{l-1}}{(l-1)!}$. Then, for every $k$ and $l$ satisfying \eqref{klrest} we have
$$|U_{k,l}|\cdot |V_{k,l}|\leq c_4^2 \cdot\frac{n^{\alpha+\beta}}{(\log n)^2}\cdot d.$$
The pair $k,l$ for which the maximum $d$ is attained satisfies $k= 2^{-2/3}\log\log n+O(1)$ and $l=2k-2$, furthermore we have $d\leq c_5(\log n)^{3/2^{2/3}}$ with some constant $c_5$.
Therefore, $|U_{k,l}|\cdot |V_{k,l}|\leq c_6 \cdot\frac{n} {(\log n)^{2-3/2^{2/3}}}$ with some $c_6$.
Therefore, by Lemma~\ref{exC6uv}. the number of edges of $H_{k,l}$ is at most
\begin{equation*}
2^{1/3}(|U_{k,l}|\cdot|V_{k,l}|)^{2/3}+16(|U_{k,l}|+|V_{k,l}|)\leq c_{7} \cdot\frac{n^{\frac{2}{3}}}{(\log n)^{4/3-2^{1/3}}}+16(n^\alpha+n^\beta).
\end{equation*}
The number of possible $(k,l)$ pairs is less than $(\log \log n)^2$, so the total number of edges of the $H_{k,l}$ graphs is at most
\begin{equation}\label{Hkl}
\sum |E(H_{k,l}|\leq c_{7}(\log\log n)^2\frac{n^{\frac{2}{3}}}{(\log n)^{4/3-2^{1/3}}}+16(\log\log n)^2(n^\alpha+n^\beta).
\end{equation}
By adding up \eqref{H1}, \eqref{H2} and \eqref{Hkl} we get that the total number of edges of $H=G_h'$ is at most
\begin{multline*}
|E(G_h')|\leq 2c_{2} \frac{n^{2/3}}{(\log n)^{0.08}} +32(n^\alpha+n^\beta)+c_{7}(\log\log n)^2\frac{n^{\frac{2}{3}}}{(\log n)^{4/3-2^{1/3}}}+\\
16(\log\log n)^2(n^{\alpha}+n^{\beta})\leq \\
\leq (c_{7}+1) \frac{n^{\frac{2}{3}}(\log\log n)^2}{(\log n)^{4/3-2^{1/3}}}+17(\log\log n)^2(n^{\frac12+\frac{h}{6K}}+n^{\frac12-\frac{h-1}{6K}}).
\end{multline*}
By summing this estimation for $1\leq h \leq K$ it is obtained that:
\begin{equation}\label{Ghbecs1}
\sum\limits_{1\leq h\leq K}|E(G_h')|\leq  c_{8} n^{\frac{2}{3}}(\log n)^{2^{1/3}-1/3}(\log\log n)^2+c_9n^{2/3}(\log\log n)^2.
\end{equation}
In the same way it can be shown that the right hand side of \eqref{Ghbecs1} is also an upper bound for the total number of edges of the $G_h''$ graphs:
\begin{equation}\label{Ghbecs2}
\sum\limits_{1\leq h\leq K}|E(G_h'')|\leq   c_{8} n^{\frac{2}{3}}(\log n)^{2^{1/3}-1/3}(\log\log n)^2+c_9n^{2/3}(\log\log n)^2.
\end{equation}

Finally, $G_{K+1}$ is also a bipartite graph, the two independent vertex classes are the primes  from the interval $(n^{2/3},n]$ and the positive integers less than $n^{1/3}$. (We delete again the vertices with degree 0.) If $p\in(n/2,n]$, then  the vertex corresponding to $p$ is the endpoint of at most one edge:  The one corresponding to
$p\cdot 1$ because  $2p>n$, so $p$ cannot be connected with an integer bigger than $1$. Delete the edges $1p$  and the vertices $p$  for $n/2<p\leq n$ from the graph 
$G_{K+1}$, and let the remaining graph be $G_{K+1}'$. Note that the number of deleted edges is at most $\pi(n)-\pi(n/2)$.
The graph $G_{K+1}'$ does not contain any hexagons either, and all of its edges join  a prime from  $(n^{2/3},n/2]$ with a positive integer less than $n^{1/3}$.
Therefore, it is a bipartite graph whose independent vertex classes $R$ and $S$ satisfy the following conditions:
$$R\subseteq \{p\ |\ n^{2/3}<p\leq n/2,p\text{ is a prime}\}\text{ and}$$
$$S\subseteq \{a\in\mathbb{N}\ |\ a<n^{1/3}\}.$$
By Lemma \ref{gy1}. for the number of edges of  $G_{K+1}'$ the inequality
$$l_{K+1}'\leq 2|R|+|S|^2/2\leq 2(\pi(n/2)-\pi(n^{2/3}))+n^{2/3}/2$$
holds. Accordingly,
\begin{equation}\label{kK+1}
l_{K+1}\leq \pi(n)-\pi(n/2)+l_{K+1}'\leq  \pi(n)+\pi(n/2)+n^{2/3}/2.
\end{equation}

Adding up the inequalities  \eqref{k0}, \eqref{Ghbecs1}, \eqref{Ghbecs2}, \eqref{kK+1}:
\begin{multline*}
l=\sum_{h=0}^{K+1} l_h\leq n^{2/3}+ 2c_{8} n^{\frac{2}{3}}(\log n)^{2^{1/3}-1/3}(\log\log n)^2+2c_9n^{2/3}(\log\log n)^2+\\
\pi(n)+\pi(n/2)+n^{2/3}/2\leq \pi(n)+\pi(n/2)+c_{10} n^{\frac{2}{3}}(\log n)^{2^{1/3}-1/3}(\log\log n)^2,
\end{multline*}
for $c_{10}=2c_9+1$ if $n$ is large enough.
Consequently,  the statement of the theorem is proved.

\section{Conclusion}
Note that Theorem~\ref{3raujtetel}. implies that for every odd $k\geq 3$ we also have an analogous upper bound for $G_k(n)$:

\begin{corollary}
Let $3\leq k$ be odd. Then $G_k(n)$ is at most $\pi(n)+\pi(n/2)+n^{\frac{2}{3}}(\log n)^{2^{1/3}-1/3+o(1)}$.


\end{corollary}

\begin{proof}
The proof of Corollary 5.1. in \cite{pach} shows how an upper bound for $G_3(n)$ extends to an upper bound for any $G_k(n)$ with odd $k$. (Note that for even values of $k$ even better upper bounds can be proved as for even $k$ we have $G_k(n)\sim \pi(n)$.)
\end{proof}

According to the lower bound in \eqref{regi3ra} and the upper bound in Theorem~\ref{3raujtetel}. we have
$$\pi(n)+\pi(n/2)+c_1n^{2/3}/(\log n)^{4/3}\leq G_3(n)\leq \pi(n)+\pi(n/2)+n^{2/3}(\log n)^{2^{1/3}-1/3+o(1)}.$$
Hence, the ratio of the error terms is $(\log n)^{2^{1/3}+1+o(1)}$. One of the reasons for this gap is  that the matching lower bound of Lemma~\ref{exC6uv}. is not known. In fact with the help of this the lower bound for $G_3(n)$ could be immediately improved to $\pi(n)+\pi(n/2)+cn^{2/3}/(\log n)^{1/3}$.
To determine the precise exponent of $\log n$ 
both the graph theoretic tools and the number theoretic factorization lemma should be improved.

Another interesting question for further research is to determine the exponent of $n$ in $G_k(n)-\pi(n)-\pi(n/2)$ for odd $k>3$, or at least to decide whether this exponent is still $2/3$ (as for $k=3$) or smaller.

\bigskip

\end{document}